\documentclass[12pt]{amsart}
\usepackage{latexsym,amsxtra,amscd,ifthen,amsmath,color, multicol,hyperref}
\usepackage{amsfonts}
\usepackage{amsthm}
\usepackage{amssymb}
\usepackage[all,cmtip]{xy}
\usepackage{tikz}
\usetikzlibrary{arrows}
\numberwithin{equation}{section}

\theoremstyle{plain}
\newtheorem{theorem}{Theorem}[section]
\newtheorem{thm}[theorem]{Theorem}
\newtheorem{lemma}[theorem]{Lemma}

\newtheorem{cor}[theorem]{Corollary}

\theoremstyle{definition}
\newtheorem{definition}[theorem]{Definition}

\newtheorem{remark}[theorem]{Remark}

\makeatletter              
\let\c@equation\c@theorem  
\makeatother

\DeclareMathOperator{\Coder}{Coder}

\DeclareMathOperator{\Ext}{Ext}

\newcommand{\C}{{\mathcal{C}}}
\newcommand{\E}{{\mathcal{E}}}
\newcommand{\Id}{{\rm Id}}

\newcommand{\Inn}{{\rm Inn}}

\newcommand{\cExt}{{\mathcal{E}} \! {\it{xt}}}

\newcommand{\one}{\mathbf{1}}

\newcommand{\kk}{\mathbf{k}}

\DeclareMathOperator{\Hom}{Hom}

\DeclareMathOperator{\Ima}{Im}
\DeclareMathOperator{\im}{Im}
 
\DeclareMathOperator{\Ker}{Ker}

\newcommand{\DOT}{\begin{picture}(2.5,2)
               (1,1)\put(2,3){\circle*{2}}\end{picture}}

\newcommand{\bu}{\DOT}

\newcommand{\ot}{\otimes}

\newcommand{\del}{\partial}

\begin{document}

\title[Lie structure]
{Graded Lie structure on cohomology of 
some exact monoidal categories
}

\author{Y.\ Volkov}
\address{Department of Mathematics and Mechanics, Saint Petersburg State
University, Saint Petersburg, Russia}
\email{wolf86\underline{\hspace{.15cm}}666@list.ru}
\author{S.\ Witherspoon}
\address{Department of Mathematics, Texas A\&M University, College Station, Texas 77843, USA}
\email{sjw@tamu.edu}

\date{30 July 2023}

\thanks{The first author was supported by RFBR according to
the research project 20-01-00030 and in part by
a Young Russian Mathematics Award.
The second author was partially supported by 
NSF grants DMS-1401016 and DMS-1665286.}

\bibliographystyle{abbrv}   

\begin{abstract}
For some exact monoidal categories,
we describe explicitly a connection between 
topological and algebraic definitions of the Lie bracket
on the extension algebra of the unit object. 
The topological definition, due to Schwede and to Hermann,
involves loops in extension categories.
The algebraic definition, due to the first author, 
involves homotopy liftings of maps. 
As a consequence of our description, 
we prove that the topological definition indeed 
yields a Gerstenhaber algebra structure
in this monoidal category setting.
This answers a question of Hermann
for those exact monoidal categories in which the unit object
has a particular type of resolution that is called power flat. 
For use in proofs, we generalize $A_{\infty}$-coderivation
and homotopy lifting techniques 
from bimodule categories to these exact monoidal categories. 
\end{abstract} 



\maketitle

\section{Introduction}

The Lie structure on Hochschild cohomology of an algebra is more difficult to
understand than is the associative algebra structure. 
There are 
fewer techniques available for handling it in relation to 
arbitrary resolutions 
or to arbitrary extensions of modules.
A topological approach introduced by Schwede~\cite{Schwede} and expanded
to some types of monoidal categories by Hermann~\cite{Hermann2} expresses
the bracket as a loop in an extension category. 
Shoikhet~\cite{Shoikhet1,Shoikhet2} and Lowen and Van den Bergh~\cite{LowVan}
offered related advances in the direction of Deligne's Conjecture. 
An algebraic approach introduced by Negron and the 
authors~\cite{NW1,Volkov} describes the bracket on an 
arbitrary projective resolution via homotopy lifting functions~\cite{Volkov}
which were expanded to $A_{\infty}$-coderivations~\cite{NVW}, providing further insight
and theoretical tools. 

In this paper, we generalize the algebraic approach 
of homotopy liftings and $A_{\infty}$-coderivations 
from the Hochschild cohomology of algebras
to the cohomology of exact monoidal categories
in which the unit object has a particular type of resolution
that is called power flat (see Definition~\ref{def:powerflat}).
We use these techniques to 
make a direct connection to the work of Schwede and Hermann.
Specifically, the topological definition of the bracket is
a loop traversing four incarnations of cup product: tensor product in each
of two orders and Yoneda splice in each of two orders. 
This definition calls on an isomorphism from homotopy classes of loops on
a category of $n$-extensions to a category of $(n-1)$-extensions,
given by Retakh and by Neeman~\cite{NeRe,Retakh}. 
The algebraic definition of the bracket via homotopy liftings 
then essentially provides a
homotopy between two resulting paths from the Yoneda splice in one
order to that in the other.  
As a consequence of this explicit description and connection
with topology, we  prove that Hermann's bracket in a
monoidal category setting indeed induces a
Gerstenhaber algebra structure on cohomology, 
answering~\cite[Question~5.2.15]{Hermann2}
for exact monoidal categories in which the unit object
has a power flat resolution.

We begin 
in Section~\ref{sec:exact} by recalling some standard definitions and
notation for exact categories and $n$-extensions. 
We then summarize some of Retakh's work on loops in
extension categories in Section~\ref{sec:SH},
in particular Schwede's and Hermann's formulation
of his work in view of its application to Lie structures.
In Section~\ref{sec:Gb} we generalize the $A_{\infty}$-coalgebra
techniques of~\cite{NVW} and the homotopy lifting techniques of~\cite{Volkov} 
to exact monoidal categories in which the unit object has a
power flat resolution,
defining a bracket on the extension algebra of the unit object
that makes it a Gerstenhaber algebra.
Finally, we make a direct connection to Schwede's
and Hermann's topological approach
in Section~\ref{sec:equiv}.

\section{Exact categories and extensions}\label{sec:exact}

In this section we recall definitions and basic facts, and we 
introduce some notation concerning exact categories and $n$-extensions.

\begin{definition} Let $\C$ be an additive category and 
$\mathcal{E}$ a class of distinguished sequences $X\rightarrow Y\rightarrow Z$ 
of $\C$.
We call $\E$ a class of {\it conflations} if for every  
sequence $X\xrightarrow{\iota} Y\xrightarrow{\pi} Z$ in $\E$,  
the morphism $\iota$ is a kernel of $\pi$ and the morphism $\pi$ is a cokernel of $\iota$. 
A morphism $\iota: X\rightarrow Y$ in $\C$ 
is an {\it inflation} if there exists a conflation of the form $X\xrightarrow{\iota} Y\xrightarrow{\pi} Z$.
A morphism $\pi:Y\rightarrow Z$ in $\C$ is a {\it deflation} if there exists a conflation of the form $X\xrightarrow{\iota} Y\xrightarrow{\pi} Z$. The pair $(\mathcal{C},\mathcal{E})$ is called an {\it exact category} if the following properties hold:
\begin{enumerate}
\item $0\rightarrow 0\rightarrow 0$ is a conflation;
\item the composition of any two deflations is also a deflation;
\item if $\pi: Y\rightarrow Z$ is a deflation and $f:Y'\rightarrow Z$ is any morphism, then there exists a pullback \begin{tikzpicture}[node distance=0.9cm]
 \node(A) {\tiny$K$};
 \node(B) [right  of=A] {\tiny$Y'$};
\node(C) [below  of=A] {\tiny$Y$};
\node(D) [below  of=B] {\tiny$Z$};

\draw [->,>=stealth'] (A)  to node[above]{\tiny$\pi'$ } (B) ; 
\draw [->,>=stealth'] (A)  to node[left]{\tiny$f'$ } (C) ; 
\draw [->,>=stealth'] (B)  to node[right]{\tiny$f$ } (D) ; 
\draw [->,>=stealth'] (C)  to node[above]{\tiny$\pi$ } (D) ;
\end{tikzpicture} with deflation $\pi'$;
\item if $\iota: X\rightarrow Y$ is an inflation and $g:X\rightarrow Y'$ is any morphism, then there exists a pushout
\begin{tikzpicture}[node distance=0.9cm]
 \node(A) {\tiny$X$};
 \node(B) [right  of=A] {\tiny$Y$};
\node(C) [below  of=A] {\tiny$Y'$};
\node(D) [below  of=B] {\tiny$R$};

\draw [->,>=stealth'] (A)  to node[above]{\tiny$\iota$ } (B) ; 
\draw [->,>=stealth'] (A)  to node[left]{\tiny$g$ } (C) ; 
\draw [->,>=stealth'] (B)  to node[right]{\tiny$g'$ } (D) ; 
\draw [->,>=stealth'] (C)  to node[above]{\tiny$\iota'$ } (D) ;
\end{tikzpicture} with inflation $\iota'$.
\end{enumerate}
\end{definition}

\begin{remark} One can show (see \cite[Appendix~A]{exKeller}) that if $(\mathcal{C},\mathcal{E})$ is an exact category, then any split exact sequence is a conflation and the composition of any two inflations is an inflation. Moreover, if $\iota$ has a cokernel and $f\iota$ is an inflation for some $f$, then $\iota$ is an inflation itself and dually if  $\pi$ has a kernel and $\pi g$ is a deflation for some $g$, then $\pi$ is a deflation. One can also show 
(see \cite[Theorem A.1 and Remark A.3]{exact}) that any extension closed full subcategory of an abelian category is exact and that any small exact category can be realized as an extension closed full subcategory of some abelian category.
\end{remark}

We will usually omit the notation $\mathcal{E}$ and call $\mathcal{C}$ an exact category meaning that there is some fixed class of conflations for $\mathcal{C}$. We are going to follow the approach of \cite{Hermann2} to the study of homological properties of exact categories. Namely, we will study the categories of $n$-extensions in $\mathcal{C}$.

\begin{definition} A sequence $\cdots\xrightarrow{d_1}E_1\xrightarrow{d_0}E_0$ with a morphism $\mu_E:E_0\rightarrow X$ is called a {\it resolution} of $X\in\C$ if there are conflations $K_0\xrightarrow{\iota_0}E_0\xrightarrow{\mu_E}X$; $K_1\xrightarrow{\iota_1}E_1\xrightarrow{\pi_0}K_0$; $\ldots$ such that $d_i=\iota_i\pi_i$ for all $i\ge 0$. In this case we will denote the corresponding resolution by $(E,d,\mu_E)$ where $E=(E_i)_{i\ge 0}$, $d=(d_i)_{i\ge 0}$.
\end{definition}

Of course, resolutions are particular cases of complexes, i.e. of sequences $\cdots\xrightarrow{d_{i+1}}E_{i+1}\xrightarrow{d_i}E_i\xrightarrow{d_{i-1}}E_{i-1}\xrightarrow{d_{i-2}}\cdots$ such that $d_id_{i+1}=0$ for all $i\in\mathbb{Z}$. Such a complex we denote by $(E,d)$. If $(E',d')$ is another complex, then a {\it degree $n$ morphism} from $(E, d)$ to $(E', d')$ is a sequence of maps $f=(f_i)_{i\in\mathbb{Z}}$ with $f_i\in\Hom_\C(E_i,E'_{i-n})$. In particular, $d$ is a degree one morphism from $(E, d)$ to itself. Degree $n$ morphisms between two fixed complexes form an abelian group in an obvious way.
Moreover, if $g$ is a degree $m$ morphism from $(E',d')$ to $(E'', d'')$, then we define the composition $gf$ as the degree $(n+m)$ morphism from $(E, d)$ to $(E'', d'')$ defined by the equality $(gf)_i = g_{i-n}f_i$ for all $i$.
For a degree $n$ morphism $f$ as above, we denote by $\del(f)$ the degree $(n + 1)$ morphism defined by the equality $\del(f) =d'f - (-1)^nf d$.
We will call $f$ a {\it chain map} if $\del(f) = 0$ and we will say that a degree $n$ morphism $f'$ from $(E, d)$ to $(E', d')$ is {\it homotopic} to $f$ and write $f'\sim f$ if $f-f' = \del(s)$ for some degree $(n-1)$ morphism $s$. We will call $f$ {\it null homotopic} if $f\sim 0$. Any object $X$ of $\C$ we will consider also as a complex $(\tilde X,0)$ with $\tilde X_0=X$ and $\tilde X_i=0$ for $i\not=0$. For two resolutions $(E, d, \mu_E)$ and $(E', d', \mu_{E'})$ of $X$, we will call a degree zero chain map $f$ from $(E, d)$ to $(E', d')$ a {\it morphism of resolutions} if it lifts the identity morphism on $X$, i.e.\ if $\mu_{E'}f=\mu_E$. If the other data is clear from the context, we will sometimes denote the complex $(E, d)$ or even the resolution $(E, d, \mu_E)$ simply by $E$.

\begin{definition} A resolution $(E, d, \mu_E)$ of $X$ is called an {\em $n$-extension} of $X$ by $Y$ if $E_n = Y$ and $E_i = 0$ for $i > n$. The class of all $n$-extensions of $X$ by $Y$ is denoted by $\cExt^n_\C(X, Y )$. As is usual, we set $\cExt^0_\C(X, Y )=\Hom_\C(X, Y )$, but this paper concerns more the case $n \ge 1$ and one may assume throughout that $n\geq 1$ whenever some argument or construction does not work for $n = 0$.
\end{definition}

For an $n$-extension $(E, d, \mu_E)$ of $X$ by $Y$, we introduce some special morphisms. We set
$\iota_E = d_{n-1}:Y\rightarrow  E_{n-1}$ and introduce morphisms $$\kappa_E:Y\rightarrow E \ \mbox{ and } \ \pi_E : E\rightarrow Y$$
of degrees $-n$ and $n$ respectively, that are identity maps in their unique nonzero degrees. Note that $\pi_E$ is a chain map, $\del(\kappa_E) = \iota_E\kappa_E$ and $\pi_E\kappa_E = 1_Y$.

Let us pick $(E, \phi, \mu_E),(F, \psi, \mu_F )\in \cExt^n_\C(X, Y )$.
A {\it morphism of $n$-extensions} from $E$ to $F$ is a morphism of resolutions $f : E\rightarrow F$ that is the identity map in degree $n$, i.e. such that $\pi_F f = \pi_E$. There are only identity morphisms between elements of $\cExt^0_\C(X, Y )$.
Morphisms generate an equivalence relation on $\cExt^n_{\C}(X,Y)$ and the
set of equivalence classes is denoted $\Ext^n_{\C}(X,Y)$. 
Also, with these morphisms, 
$\cExt^n_\C(X, Y )$ is turned into a category for any $n \ge 0$. Then one can define homotopy groups $\pi_i\cExt^n_\C(X, Y )$ of $\cExt^n_\C(X, Y )$ as homotopy groups of the classifying
space $\mathcal{B}\big(\cExt^n_\C(X, Y )\big)$. For a more direct interpretation of the groups $\pi_0\cExt^n_\C(X, Y )$ and $\pi_1\cExt^n_\C(X, Y )$
one can look, for example, at \cite[\S2.2]{Hermann2}.
 In particular, $\Ext^n_\C(X, Y ) = \pi_0\cExt^n_\C(X, Y )$ consists of the classes of elements of $\cExt^n_\C(X, Y )$ modulo the minimal equivalence relation such that $E$ is equivalent to $F$ whenever
 $\Hom_{\cExt^n_\C(X,Y )}(E, F)\not=\varnothing$.
 
 Let us now recall some constructions involving $n$-extensions. First of all, let us pick $(E, \phi, \mu_E) \in \cExt^n_\C(X, Y )$ and two morphisms $\alpha : X'\rightarrow X$ and $\beta : Y\rightarrow Y'$.
 Then we define $E\alpha = (E\alpha, \phi^{\alpha}, \mu_{E\alpha}) \in \cExt^n_\C(X', Y )$ and $\beta E = (\beta E, {}^\beta\phi, \mu_{\beta E}) \in \cExt^n_\C(X, Y')$ in the
following way. Let us construct\\
the pullback \begin{tikzpicture}[node distance=0.8cm]
 \node(A) {\tiny$(E\alpha)_0$};
 \node(Ar) [right  of=A] {};
 \node(B) [right  of=Ar] {\tiny$X'$};
\node(C) [below  of=A] {\tiny$E_0$};
\node(D) [below  of=B] {\tiny$X$};

\draw [->,>=stealth'] (A)  to node[above]{\tiny$\mu_{E\alpha}$ } (B) ; 
\draw [->,>=stealth'] (A)  to node[left]{\tiny$\bar\alpha$ } (C) ; 
\draw [->,>=stealth'] (B)  to node[left]{\tiny$\alpha$ } (D) ; 
\draw [->,>=stealth'] (C)  to node[above]{\tiny$\mu_E$ } (D) ;
\end{tikzpicture} of $\mu_E$ along $\alpha$ and the pushout \begin{tikzpicture}[node distance=0.8cm]
 \node(A) {\tiny$Y$};
 \node(Ar) [right  of=A] {};
 \node(B) [right  of=Ar] {\tiny$E_{n-1}$};
\node(C) [below  of=A] {\tiny$Y'$};
\node(D) [below  of=B] {\tiny$(\beta E)_{n-1}$};

\draw [->,>=stealth'] (A)  to node[above]{\tiny$\iota_E$ } (B) ; 
\draw [->,>=stealth'] (A)  to node[left]{\tiny$\beta$ } (C) ; 
\draw [->,>=stealth'] (B)  to node[left]{\tiny$\bar\beta$ } (D) ; 
\draw [->,>=stealth'] (C)  to node[above]{\tiny$\iota_{\beta E}$ } (D) ;
\end{tikzpicture} of $\iota_E$ along $\beta$.
Now we set $(E \alpha)_i = E_i$, $\phi^{\alpha}_i = \phi_i$ for $i > 0$ and define $\phi^{\alpha}_0$ as the unique morphism such that $\mu_{E_\alpha}\phi^{\alpha}_0=0$ and $\bar\alpha \phi^{\alpha}_0 = \phi_0$.
We set also $(\beta E)_i = E_i$, ${}^\beta \phi_{i-1} = \phi_{i-1}$ for $i < n - 1$, $\mu_{\beta E} = \mu_E$
and define ${}^\beta \phi_{n-2}$ as the unique morphism such that ${}^\beta \phi_{n-2}\iota_{\beta E} = 0$ and ${}^\beta \phi_{n-2}\bar\beta= \phi_{n-2}$.
In the case $n = 1$ the last construction must be slightly corrected, because in this case the pushout construction must be applied to $\mu_{\beta E}\not=\mu_E$.
One can see that $(\beta E)\alpha = \beta(E\alpha)$ and so the notation $\beta E\alpha$ makes sense. Let us now pick two extensions $E, F\in \cExt^n_\C(X, Y )$. We
define their sum (called the {\it Baer sum}) in the following way. First we form the $n$-extension $E\oplus F$ of $X^2$ by $Y^2$ in the obvious way and then define $E + F =
\begin{pmatrix}1&1\end{pmatrix}(E \oplus F)\begin{pmatrix}1\\ 1\end{pmatrix}$.
 This sum operation determines a commutative monoid structure on the set of isomorphism classes of $n$-extensions of $X$ by $Y$.  The zero element for this operation is
 \begin{equation}\label{sigma}
 \sigma_n(X, Y ) =\left(0 \rightarrow Y\xrightarrow{1_Y} Y \rightarrow 0 \rightarrow \cdots \rightarrow 0 \rightarrow X\right)\in \cExt^n_\C(X, Y)    
 \end{equation}
 with $\mu_{\sigma_n(X,Y )} = 1_X$, where for $n = 1$ the middle terms $Y$ and $X$ glue together and form the direct sum $X \oplus Y$.  Moreover, the sum operation passes to $\Ext^n_\C(X, Y )$
and determines the structure of an abelian group on it.
 If the underlying category $\C$ is $\kk$-linear for some commutative ring $\kk$, then $\Ext^n_\C(X, Y )$ is a $\kk$-module, where the $n$-extension $aE = Ea$ is defined via the identification of $a\in\kk$ with the morphism
$a1_X : X \rightarrow X$. 
In particular, if $a$ is invertible and $E$ is reserved for $(E, \phi, \mu_E)$, then $aE$ denotes the $n$-extension $(E, \phi, a^{-1}\mu_E)$.
 We set $\Ext^{\bu}_\C(X, Y ) = \oplus_{n\ge 0}\Ext^n_\C(X, Y )$. Note that at this moment this definition makes sense.

Let us pick now $(E, \phi, \mu_E) \in  \cExt^n_\C(X, Y )$ and $(F, \psi, \mu_F ) \in \cExt^m_\C(Y, Z)$. We define the $(m + n)$-extension $F\#E$ as the Yoneda splice
\[ 
0  \rightarrow Z \xrightarrow{\iota_F}F_{m-1} \xrightarrow{\psi_{m-2}}\cdots  \xrightarrow{\psi_0}F_0 \xrightarrow{\iota_E\mu_F}E_{n-1} \xrightarrow{\phi_{n-2}}\cdots  \xrightarrow{\phi_0}E_0
\]
with $\mu_{F\#E} = \mu_E$. This construction passes to the sets $\Ext_\C$, 
i.e.\ it induces a product $\#: \Ext^m_\C(Y, Z) \times \Ext^n_\C(X, Y )\rightarrow \Ext^{m+n}_\C(X, Z)$ which is called the {\it Yoneda product}.
In particular, for any object $X$ of $\C$ the set $\Ext^{\bu}_\C(X, X)$ is a ring with respect to operations $+$ and $\#$. If $\C$ is $\kk$-linear, then $\Ext^n_\C(X, X)$ is a $\kk$-algebra.

 Note that one can define the derived category ${\rm D}\mathcal{C}$ of the exact category $\mathcal{C}$ (see, for example, \cite[\S10.4]{exact}).
 Given an $n$-extension $(E, \phi, \mu_E)$ of $X$ by $Y$, one can define a morphism from $X$ to $Y [n]$ in ${\rm D}\C$ as the composition $\pi_E\mu_E^{-1}$ which makes sense because $\mu_E$ is a quasi isomorphism. This correspondence induces an isomorphism between $\Ext_{\mathcal{C}}^n(X,Y)$ and $\Hom_{{\rm D}\mathcal{C}}(X,Y[n])$ that respects the additive ($\kk$-linear) structure and sends the Yoneda product of two sequences to the composition of the corresponding morphisms in the derived category in the sense that $\pi_{F\#E}\mu_{F\#E}^{-1}$ coincides with $(\pi_F\mu_F^{-1})[n]\pi_E\mu_E^{-1}$ up to a sign.  This gives a strong motivation to study the groups $\Ext_{\mathcal{C}}^n(X,Y)$.
 
 If the category $\C$ satisfies an additional property, namely, if it has {\it enough projective objects}, then the groups $\Ext^n_\C(X, Y )$ have another, more usable, description.
 
 \begin{definition} The object $P$ of an exact category $\C$ is called {\it projective} if any deflation $X\rightarrow P$ is a split epimorphism. The resolution $(P, d, \mu_P )$ of $X \in \C$ is called projective if $P_i$ is projective for each $i \ge 0$.
 \end{definition}
 
 If $X$ has a projective resolution $(P, d, \mu_P )$, standard arguments show that there exists a canonical isomorphism of abelian groups ($\kk$-spaces if $\C$ is $\kk$-linear) $\Ext^n_\C(X, Y )\cong\Ker \Hom_\C(d_{n}, Y )/ \Ima \Hom_\C(d_{n-1}, Y )$.
Moreover, the Yoneda product on the left side of this isomorphism can be calculated on the right side via a standard lifting technique. In this paper we will restrict ourselves to the case of $n$-extensions of objects $X\in\C$ having projective
resolutions.

Let us recall the construction of the isomorphism of abelian groups $\Ext^n_\C(X, Y )\cong\Ker \Hom_\C(d_{n}, Y )/ \Ima \Hom_\C(d_{n-1}, Y )$. Let us first pick some {\it $n$-cocycle}, i.e.\ a degree $n$ chain map $f : P\rightarrow Y$.
We denote by $K(f)$ the element
\begin{equation}\label{Kf}
    0 \rightarrow Y \xrightarrow{\iota_f}K(f)_{n-1}\xrightarrow{d_f} P_{n-2}\xrightarrow{d_{n-3}}\cdots \xrightarrow{d_0}P_0
\end{equation}
of $\cExt^n_\C(X, Y )$ with $\mu_{K(f)} = \mu_P$,  where $K(f)_{n-1}$ is the pushout of the morphisms $d_{n-1}:P_n\rightarrow P_{n-1}$ and $f:P_n\rightarrow Y$. To construct this pushout, one first factors $d_{n-1}$ as $P_n\xrightarrow{\pi_{n-1}}K_{n-1}\xrightarrow{\iota_{n-1}}P_{n-1}$ where $\pi_{n-1}$ is the cokernel of $d_{n}$ and then constructs the pushout of the inflation $\iota_{n-1}$ along the unique morphism $\bar f$ such that $f=\bar f\pi_{n-1}$.
We denote the remaining arrow of this pushout by 
$$\theta_f:P_{n-1}\rightarrow K(f)_{n-1} .$$
  The morphism $d_f$ arises as the unique morphism such that $d_f\iota_f = 0$ and $d_f \theta_f = d_{n-2}$.
In the case $n = 1$ this construction has to be slightly corrected via applying the pushout construction to obtain $\mu_{K(f)}\not= \mu_P$. The map from $\Ker \Hom_\C(d_{n}, Y )$ to $\cExt^n_\C(X, Y)$ sending $f$ to $K(f)$ induces the required isomorphism. The inverse to this isomorphism can be constructed in the following way. For any $n$-extension $(E, \phi, \mu_E)$ of $X$ by $Y$, there exists a morphism of resolutions $\hat f : P\rightarrow E$.
Then the map from $\cExt^n_\C(X, Y )$ to $\Ker \Hom_\C(d_{n}, Y )$ sending $E$ to $\hat f_n = \pi_E\hat f$ for some morphism of resolutions $\hat f$ induces the required inverse isomorphism not depending on the choice of $\hat f$.

\section{Schwede's and Hermann's formulas for Retakh's isomorphism}
\label{sec:SH}

An important feature of homotopy groups of extensions is the isomorphism $\Ext_{\mathcal{C}}^{n-i}(X,Y)\cong \pi_i\cExt_{\mathcal{C}}^n(X,Y)$ proved in \cite[Theorem 1]{Retakh} for an abelian category and in \cite[Theorem 5.2]{NeRe} for a Waldhausen category. In \cite[\S3.1]{Hermann2} the isomorphism $\Ext_{\mathcal{C}}^{n-1}(X,Y)\cong \pi_1\cExt_{\mathcal{C}}^n(X,Y)$ was established explicitly when $\mathcal{C}$ is a factorizing exact category. Let us recall the definition of a factorizing  exact category given in \cite[\S2.1]{Hermann2}.
Suppose that $(E, \phi, \mu_E),(F, \psi, \mu_F )\in \cExt_{\mathcal{C}}^n(X,Y)$ and $\beta:E\rightarrow F$ is a morphism of $n$-extensions. Let $\hat F$ be the $n$-extension of $X$ by $Y$ defined by the sequence
\begin{multline*}
Y\xrightarrow{\tiny\begin{pmatrix}\iota_F\\0\end{pmatrix}}F_{n-1}\oplus E_{n-2}\xrightarrow{\tiny\begin{pmatrix}\psi_{n-2}&0\\0&1\\0&0\end{pmatrix}}F_{n-2}\oplus E_{n-2}\oplus E_{n-3}\\
\xrightarrow{\tiny\begin{pmatrix}\psi_{n-3}&0&0\\0&0&1\\0&0&0\end{pmatrix}}F_{n-3}\oplus E_{n-3}\oplus E_{n-4}\rightarrow\cdots
\rightarrow F_1\oplus E_1\oplus E_0\xrightarrow{\tiny\begin{pmatrix}\psi_{0}&0&0\\0&0&1\end{pmatrix}}F_0\oplus E_0
\end{multline*}
and the morphism $\mu_{\hat F}=\begin{pmatrix}
\mu_F&0 
\end{pmatrix}: F_0 \oplus E_0\rightarrow X$.
Let us define $\hat\beta\in \Hom_{\cExt_{\mathcal{C}}^n(X,Y)}(E,\hat F)$ degreewise. We set
$$
\hat\beta_0=\begin{pmatrix}\beta_0\\ 1_{E_0}\end{pmatrix},\,
 \ \hat\beta_i=\begin{pmatrix}\beta_i\\ 1_{E_i}\\\phi_{i-1}\end{pmatrix}\,(1\le i\le n-2),\, \ \hat\beta_{n-1}=\begin{pmatrix}\beta_{n-1}\\ \phi_{n-2}\end{pmatrix}.
$$
Due to \cite[Definition 2.1.11]{Hermann2}, the exact category $\mathcal{C}$ is called {\em factorizing} if all components of $\hat\beta$ are inflations for any $n\ge 1$, $X,Y\in\mathcal{C}$, any $E, F \in \cExt^n_\C(X, Y )$, and any $\beta\in\Hom_{\cExt^n_\C(X,Y)}(E, F)$.

The next lemma, stating that any exact category is factorizing,
implies that many results of \cite{Hermann2} hold generally
for all exact categories.

\begin{lemma}\label{lem:factorizing}
  Any exact category is factorizing.
\end{lemma}
\begin{proof} It is easy to see that $\hat\beta_0$ is a split monomorphism with the cokernel $F_0$ and $\hat\beta_i$ ($1\le i\le n-2$) is a split monomorphism with the cokernel $F_i\oplus E_{i-1}$. Thus, it remains to prove that $\hat\beta_{n-1}$ is an inflation.
To do this, let us first present $\phi_{n-2}$ and $\psi_{n-2}$ in the form $\phi_{n-2}=\iota_\phi\pi_\phi$ and $\psi_{n-2}=\iota_\psi\pi_\psi$, where $Y\xrightarrow{\iota_E}E_{n-1}\xrightarrow{\pi_\phi}K_\phi$ and $Y\xrightarrow{\iota_F}F_{n-1}\xrightarrow{\pi_\psi}K_\psi$ are conflations. We have $\hat\beta_{n-1}=\begin{pmatrix}\beta_{n-1}\\ \phi_{n-2}\end{pmatrix}=\begin{pmatrix}1_{F_{n-1}}&0\\0& \iota_\phi\end{pmatrix}\begin{pmatrix}\beta_{n-1}\\ \pi_\phi\end{pmatrix}$. Note that $\begin{pmatrix}1_{F_{n-1}}&0\\0& \iota_\phi\end{pmatrix}$ is an inflation, for example, as a pushout of the inflation $\iota_\phi$ along the direct inclusion of $K_\phi$ to $F_{n-1}\oplus K_\phi$, and hence it remains to prove that $\begin{pmatrix}\beta_{n-1}\\ \pi_\phi\end{pmatrix}$ is an inflation.

Note that by the cokernel universal property there exists $\gamma:K_\phi\rightarrow K_\psi$ such that $\gamma \pi_\phi=\pi_\psi\beta_{n-1}$. Then $(1_Y,\beta_{n-1},\gamma)$ is a morphism of short exact sequences, and hence the square
\begin{tikzpicture}[node distance=0.8cm]
 \node(A) {\tiny$E_{n-1}$};
 \node(Ar) [right  of=A] {};
 \node(B) [right  of=Ar] {\tiny$K_\phi$};
\node(C) [below  of=A] {\tiny$F_{n-1}$};
\node(D) [below  of=B] {\tiny$K_\psi$};

\draw [->,>=stealth'] (A)  to node[above]{\tiny$\pi_\phi$ } (B) ; 
\draw [->,>=stealth'] (A)  to node[left]{\tiny$\beta_{n-1}$ } (C) ; 
\draw [->,>=stealth'] (B)  to node[left]{\tiny$\gamma$ } (D) ; 
\draw [->,>=stealth'] (C)  to node[above]{\tiny$\pi_\psi$ } (D) ;
\end{tikzpicture}
is a pullback of the deflation $\pi_\psi$. Now it follows from \cite{exKeller} that $\begin{pmatrix}\beta_{n-1}\\ \pi_\phi\end{pmatrix}$ is an inflation and we are done.
\end{proof}

\begin{cor}\label{monoRek} For any exact category $\mathcal{C}$ there exists an isomorphism $\gamma:\Ext_{\mathcal{C}}^{n-1}(X,Y)\cong \pi_1\cExt_{\mathcal{C}}^n(X,Y)$, explicitly constructed in \cite[\S3.1]{Hermann2}.
\end{cor}

The isomorphism of Corollary \ref{monoRek} was used by Hermann \cite[\S5.2]{Hermann2} to define the Gerstenhaber bracket on the extension algebra of the unit of an exact monoidal category. It was first constructed explicitly by Schwede \cite[\S2]{Schwede} for any category of modules. Hermann showed that for a module category his construction coincides up to a sign with that of Schwede, and hence the bracket on Hochschild cohomology constructed by Hermann coincides with the usual Gerstenhaber bracket. The construction of the required isomorphism was done in \cite[Theorem 3.1]{Schwede} using projective resolutions and for this reason is more appropriate for us. Now we will  show that if $X$ has a projective resolution, then Schwede's isomorphism coincides up to a sign with Hermann's isomorphism, generalizing \cite[Theorem 5.3.2]{Hermann2} to monoidal categories with enough projectives. 

Let us first adapt Schwede's construction to the setting of an arbitrary exact category to construct the isomorphism
\begin{equation}\label{eqn:mu}
  \mu : \Ext^{n-1}_\C(X, Y )\rightarrow \pi_1\cExt^n_\C(X, Y )
\end{equation}
in the case where $X$ has a projective resolution $(P, d, \mu_P )$. Let us fix some $n$-cocycle $f : P\rightarrow Y$ and define $K(f)\in \cExt^n_\C(X, Y )$ as in \eqref{Kf}. Note that any element of $\Ext^n_\C(X, Y )$ can be represented by $K(f)$ for some $n$-cocycle $f$.
Let now $g : P\rightarrow Y$ be an $(n - 1)$-cocycle. The pushout universal property ensures existence of a unique morphism $h : K(f)_{n-1}\rightarrow K(f)_{n-1}$ such that $h\theta_f = \theta_f - \iota_fg$ and $h\iota_f = \iota_f$.
This gives the morphism of $n$-extensions 
$$\mu_f (g) : K(f)\rightarrow K(f)$$ 
that is the identity in all degrees except $(n - 1)$ where it equals $h$. The morphism $\mu_f (g)$ determines an element of $\pi_1\cExt^n_\C(X, Y )$.
The homotopy class of $\mu_f (g)$ is determined by the cohomology class of $g$. This follows from \cite[Lemma 3.2.4]{Hermann2} because, for a degree $(n- 2)$ morphism $p : P\rightarrow Y$ , the degree $-1$ morphism from $K(f)$ to $K(f)$
that equals zero in all degrees except $(n-2)$, where it equals $\iota_f p$, is a homotopy between $\mu_f (g)$ and $\mu_f (g+pd)$.
Moreover, it is not difficult to see that $\mu_f(g_1 + g_2) = \mu_f (g_2) \circ \mu_f (g_1)$, and hence the image of $\mu_f (-)$ is an
abelian subgroup of $\pi_1\big(\cExt^n_\C(X, Y ),K(f)\big)$. A little later we will show that $\mu_f(-)$ is an isomorphism, which will ensure that $\pi_1\cExt^n_\C(X, Y )$ does not depend (up to unique isomorphism) on a
point in a connected component. Moreover, our arguments will imply that this unique isomorphism sends $\mu_{f_1}(g)$ to $\mu_{f_2}(g)$ if $f_1$ and $f_2$ are cohomologous.
For now we choose for each point $E \in \cExt^n_\C(X, Y )$ a morphism of resolutions $\hat f:P\rightarrow E$ and define $$\mu_E(g):E\rightarrow E$$ to be the conjugation of
$\mu_f(g)$, where $f = \pi_E\hat f$, by the path corresponding to the morphism from $K(f)$ to $E$ induced by $\hat f$ (not caring about the dependence of $\mu_E(g)$ on the choice of $\hat f$).

Suppose now that we have two morphisms $\alpha,\beta:K(f)\rightarrow E$ for some $(E,\phi,\mu_E)\in\cExt_{\mathcal{C}}^{n}(X,Y)$. We will show how one can recover an $(n-1)$-cocycle $g$ such that $\mu_f(g)$ is homotopic to the loop $\alpha^{-1}\beta$. Note that in fact any loop with the base point $K(f)$ can be put into this form due to the results of \cite{Hermann2,Schwede} and so we will be able to recover a preimage of any loop. Our construction imitates, of course, the construction of Schwede, but we give it for convenience, because our settings are more general.
Note that $K(f)$ comes with a canonical morphism of resolutions $\Phi_f:P\rightarrow K(f)$ defined by the equalities $(\Phi_f)_i=1_{P_i}$ for $0\le i\le n-2$, $(\Phi_f)_{n-1}=\theta_f$ and $(\Phi_f)_{n}=f$. Then $(\alpha-\beta)\Phi_f$ is a chain map that is annihilated by the quasi isomorphism $\mu_E$. Thus, this map is null homotopic, i.e.\ there is a degree $-1$ morphism $$s:P\rightarrow E$$ such that $(\alpha-\beta)\Phi_f=\phi s+sd$. Note that $\pi_Esd=0$, and hence $s_{n-1}=\pi_Es:P\rightarrow Y$ is an $(n-1)$-cocycle.

\begin{lemma}\label{preim} The loops $\mu_f(s_{n-1})$ and $\alpha^{-1}\beta$ are homotopic.
\end{lemma}
\begin{proof} Let us first replace $\beta$ by $\beta'$, where $\beta'_i=\beta_i+\phi_i s_i+s_{i-1}d_{i-1}$ for $0\le i\le n-2$,  $\beta'_{n-1}=\beta_{n-1}+s_{n-2}d_f$ and $\beta'_n=\beta_n$. Then the paths corresponding to $\beta'$ and $\beta$ are homotopic by \cite[Lemma 3.2.4]{Hermann2}. It remains to note that $\alpha\mu_f(s_{n-1})=\beta'$.
\end{proof}

Let us now recall Hermann's construction of the isomorphism 
$$\gamma:\Ext_{\mathcal{C}}^{n-1}(X,Y)\rightarrow \pi_1\cExt_{\mathcal{C}}^n(X,Y)
. $$ 
For $(F,\psi,\mu_F)\in\cExt_{\mathcal{C}}^{n-1}(X,Y)$, let us first construct a loop with a base point in the $n$-extension $\sigma_n(X,Y)$ defined by \eqref{sigma}. We  denote by $\bar F$  the $n$-extension
\[
 0\rightarrow Y \xrightarrow{\iota_F} F_{n-2}\xrightarrow{\psi_{n-3}}\cdots\xrightarrow{\psi_0} F_0\xrightarrow{\scriptsize\begin{pmatrix}\mu_F\\-\mu_F\end{pmatrix}} X^2
\]
with $\mu_{\bar F}=\begin{pmatrix}1_X&1_X\end{pmatrix}$.
There are morphisms of $n$-extensions $\alpha^F,\beta^F:\sigma_n(X,Y)\rightarrow \bar F$ both of which are equal to $\iota_F$ in degree $(n-1)$ and zero in degrees from $1$ to $(n-2)$. 
In degree zero, $\alpha^F$ equals $\begin{pmatrix}1_X\\0\end{pmatrix}$ while $\beta^F$ equals $\begin{pmatrix}0\\1_X\end{pmatrix}$. These morphisms determine the loop $(\alpha^F)^{-1}\beta^F$ that we denote by $\gamma_{\sigma_n(X,Y)}(F)$.
Now, for an arbitrary $E\in\cExt_{\mathcal{C}}^n(X,Y)$, the loop 
$$\gamma_{E}(F)\in \pi_1\cExt_{\C}^{n-1}(X,Y)$$
is obtained from the loop $\gamma_{\sigma_n(X,Y)}(F)$ by applying the functor $(-)+E$, where the plus sign denotes the Baer sum of extensions.
Since $\sigma_n(X,Y)+E=E$, we get a loop with the base point $E$. See \cite[\S3.1]{Hermann2} for details.
Hermann has shown that this construction indeed determines an isomorphism $\gamma:\Ext_{\mathcal{C}}^{n-1}(X,Y)\rightarrow \pi_1\cExt_{\mathcal{C}}^n(X,Y)$. We will show that up to a sign, the constructions of Schwede and of Hermann give the same result. This, in particular, will ensure that Schwede's construction gives a well defined isomorphism between $\Ext_{\mathcal{C}}^{n-1}(X,Y)$ and $\pi_1\cExt_{\mathcal{C}}^n(X,Y)$ in our context and will allow us to use this isomorphism 
for studying the bracket as introduced in \cite[\S5.2]{Hermann2}.

Our aim is to prove that $\mu_E(F)\sim \gamma_E\big((-1)^{n+1}F\big)$ for any $F\in\cExt_{\mathcal{C}}^{n-1}(X,Y)$ and $E\in\cExt_{\mathcal{C}}^n(X,Y)$.
Let us pick a morphism of resolutions $\hat f:P\rightarrow E$ and denote by $\bar f:K(f)\rightarrow E$ the morphism induced by it, where $f=\pi_E\hat f$.
Note that \[(u+1_E)\bar f=(u+\bar f)=(1_{\bar F}+\bar f)(u+1_{K(f)})\] for each $u\in\{\alpha^F,\beta^F\}$, and hence, by Hermann's definition, we have
\begin{multline*}
\gamma_E(F)=(\alpha^F+1_E)^{-1}(\beta^F+1_E)\sim\bar f(\alpha^F+1_{K(f)})^{-1}(1_{\bar F}+\bar f)^{-1}(1_{\bar F}+\bar f)(\beta^F+1_{K(f)})\bar f^{-1}\\
\sim \bar f(\alpha^F+1_{K(f)})^{-1}(\beta^F+1_{K(f)})\bar f^{-1}=\bar f \gamma_{K(f)}\big(F\big)\bar f^{-1}.
\end{multline*}
Thus, the required equality follows from the definition of $\mu$, our arguments above and the next lemma.

\begin{lemma}\label{SchHer} $\mu_f\big((-1)^{n+1}g\big)\sim \gamma_{K(f)}\big(K(g)\big)$ for all $n$-cocycles $f$ and $(n-1)$-cocycles $g$.
\end{lemma}
\begin{proof} We first describe the loop $\gamma_{K(f)}\big(K(g)\big)=(\alpha^{K(g)}+1_{K(f)})^{-1}(\beta^{K(g)}+1_{K(f)})$. To do this we need to compute the extension $\overline{K(g)}+K(f)$ and morphisms \[\alpha^{K(g)}+1_{K(f)}, \ \ \beta^{K(g)}+1_{K(f)}:K(f)\rightarrow \overline{K(g)}+K(f).\]
These can be obtained via a pullback-pushout construction from the morphisms of long exact sequences

\[\tiny
\begin{xy}*!C\xybox{
\xymatrix{
Y\oplus Y
  \ar[r]^{\hspace{-0.5cm}\begin{pmatrix}1_Y&0\\0&\iota_f\end{pmatrix}}
  \ar[d]^{=} 
& Y\oplus K(f)_{n-1}
  \ar[r]^{\begin{pmatrix}0&d_f\end{pmatrix}}
  \ar[d]^{\begin{pmatrix}\iota_g&0\\0&1_{K(f)_{n-1}}\end{pmatrix}} 
& P_{n-2}\ar[r]^{d_{n-3}}
  \ar[d]^{\begin{pmatrix}0\\1_{P_{n-2}}\end{pmatrix}}
&\cdots
  \ar[r]^{d_1}
&P_1
  \ar[d]^{\begin{pmatrix}0\\1_{P_{1}}\end{pmatrix}}
  \ar[r]^{\begin{pmatrix}0\\d_0\end{pmatrix}} 
& X\oplus P_0
  \ar[r]^{\begin{pmatrix}1_X&0\\0&\mu_P\end{pmatrix}}
  \ar[d]^{\delta} & X\oplus X\ar[d]^{=} \\
Y\oplus Y
  \ar[r]_{\hspace{-0.8cm}\vspace{.1cm}\begin{pmatrix}\iota_g&0\\0&\iota_f\end{pmatrix}} 
& K(g)_{n-2}\oplus K(f)_{n-1}
  \ar[r]_{\hspace{0.1cm}\vspace{.1cm}\begin{pmatrix}d_g&0\\0&d_f\end{pmatrix}} 
&P_{n-3}\oplus P_{n-2}
  \ar[r]_{\hspace{0.1cm}\vspace{.1cm}\begin{pmatrix}d_{n-4}&0\\0&d_{n-3}\end{pmatrix}}
&\cdots
  \ar[r]_{\vspace{.1cm}\begin{pmatrix}d_{0}&0\\0&d_{1}\end{pmatrix}} 
&P_0\oplus P_1
  \ar[r]_{\vspace{.1cm}\begin{pmatrix}\mu_P&0\\-\mu_P&0\\0&d_0\end{pmatrix}} 
&X\oplus X\oplus P_0
  \ar[r]_{\hspace{0.1cm}\vspace{.1cm}\begin{pmatrix}1_X&1_X&0\\0&0&\mu_P\end{pmatrix}} 
& X\oplus X
}}
\end{xy}
\]
where $\delta=\begin{pmatrix}1_X&0\\0&0\\0&1_{P_0}\end{pmatrix}$ for $\alpha^{K(g)}+1_{K(f)}$ and $\delta=\begin{pmatrix}0&0\\1_X&0\\0&1_{P_0}\end{pmatrix}$ for $\beta^{K(g)}+1_{K(f)}$. Let $K(g)_{n-2}\oplus K(f)_{n-1}\xrightarrow{\pi} L$ be the deflation completing the inflation $\begin{pmatrix}\iota_g\\-\iota_f\end{pmatrix}$ to a conflation. It is easy to see that the diagrams
\[\tiny
\begin{xy}*!C\xybox{
\xymatrix{
Y\oplus Y\ar[d]^{\begin{pmatrix}1_Y&1_Y\end{pmatrix}}\ar[rr]^{\hspace{-0.7cm}\begin{pmatrix}\iota_g&0\\0&\iota_f\end{pmatrix}} && K(g)_{n-2}\oplus K(f)_{n-1}\ar[d]^{\pi}\\
Y\ar[rr]_{\pi\begin{pmatrix}\iota_g\\0\end{pmatrix}} && L
}}
\end{xy}
\hspace{0.6cm}
\mbox{ \normalsize{and} }
\hspace{0.6cm}
\begin{xy}*!C\xybox{
\xymatrix{
X\oplus P_0\ar[rr]^{\begin{pmatrix}0&\mu_P\end{pmatrix}}\ar[dd]^{\begin{pmatrix}1_X&0\\-1_X&\mu_P\\0&1_{P_0}\end{pmatrix}} && X\ar[dd]^{\begin{pmatrix}1_X\\1_X\end{pmatrix}} \\
&\\
X\oplus X\oplus P_0\ar[rr]_{\hspace{0.2cm}\vspace{.1cm}\begin{pmatrix}1_X&1_X&0\\0&0&\mu_P\end{pmatrix}} && X\oplus X
}}
\end{xy}
\]
are a pushout and a pullback respectively. Then $\overline{K(g)}+K(f)$ is the $n$-extension
\[\tiny
\begin{xy}*!C\xybox{
\xymatrix{
Y\ar[r]^{\pi\begin{pmatrix}\iota_g\\0\end{pmatrix}} & L\ar[r]^{\hspace{-0.5cm}\overline{\begin{pmatrix}d_g&0\\0&d_f\end{pmatrix}}} &P_{n-3}\oplus P_{n-2}\ar[r]&\cdots\ar[r]&P_0\oplus P_1\ar[r]^{\begin{pmatrix}\mu_P&0\\0&d_0\end{pmatrix}}&X\oplus P_0\ar[r]^{\hspace{0.1cm}\begin{pmatrix}0&\mu_P\end{pmatrix}} & X .
}}
\end{xy}
\]
Moreover, the morphism $\Phi=\big((\alpha^{K(g)}+1_{K(f)})-(\beta^{K(g)}+1_{K(f)})\big)\Phi_f:P\rightarrow \overline{K(g)}+K(f)$ is zero in all degrees except degree zero where it equals $\begin{pmatrix}\mu_P\\0\end{pmatrix}$.
Let us define a morphism  $s:P\rightarrow \overline{K(g)}+K(f)$ of degree~$-1$ by the equalities $s_i=(-1)^i\begin{pmatrix}1_{P_i}\\0\end{pmatrix}$ for $0\le i\le n-3$, $s_{n-2}=(-1)^n\pi\begin{pmatrix}\theta_g\\0\end{pmatrix}$ and $s_{n-1}=(-1)^{n+1}g$.
It remains to note that $s$ is a homotopy for $\Phi$ and to apply Lemma \ref{preim}.
\end{proof}

\begin{remark} The proof of Lemma \ref{SchHer} can be obtained via an adaptation of the proof of \cite[Lemma 5.3.3]{Hermann2}, but we included our proof for convenience of the reader and because for us it seems to be more self-contained.
Two isomorphisms between $\Ext_{\mathcal{C}}^{n-1}(X,Y)$ and $\pi_1\cExt_{\mathcal{C}}^n(X,Y)$ were used in our proof of Lemma \ref{SchHer} (cf.\ 
\cite[Theorem 5.3.3]{Hermann2}). The second one $\gamma'$ satisfies the equality $\gamma'_E(F)=\gamma_E\big((-1)^{n+1}F\big)$ and so allows to exclude a sign from the isomorphism stated in the lemma. In fact, the proof of \cite[Lemma 5.3.3]{Hermann2} starts with passing from $\gamma$ to $\gamma'$ and the sign appears exactly at this moment. We do not know why $\gamma$ is used more in \cite{Hermann2}, but actually $\gamma$ works better with injective resolutions while $\gamma'$ is more appropriate for projective resolutions.
\end{remark}

\section{The Gerstenhaber bracket on the extension algebra of the unit}
\label{sec:Gb}

In this section we introduce our definition of the bracket on the extension algebra of the unit of an exact monoidal category when that unit has a projective power flat resolution. We then prove that, together with the Yoneda product, it gives a Gerstenhaber algebra structure. Our construction will be based on the $A_{\infty}$-coalgebra techniques of \cite{NVW}. This will allow us to obtain automatically all the desired properties, while formally the conditions required for the constructions of \cite{NVW} are redundant. Alternatively, one can
use directly the techniques of \cite{Volkov} (see also~\cite[Section~6.3]{W}) to define the bracket and then prove its properties by direct calculations using some weaker additional assumptions. In the next section we will show that under our assumptions the bracket defined in this paper coincides with the bracket introduced in \cite{Hermann2}. This allows us to prove in our setting some properties of the bracket that were left as open questions in \cite{Hermann2}.

We first recall the definition and some basic facts about monoidal categories and discuss some relations between exact and monoidal structures on a category that allow construction of the Gerstenhaber bracket on the extension algebra of the unit.

\begin{definition}\label{defn:moncat} 
Suppose that the additive category $\mathcal{C}$ is equipped with a functor $\otimes:\mathcal{C}\times \mathcal{C}\rightarrow \mathcal{C}$, a distinguished object $\one$ and natural isomorphisms of functors
$$
-\otimes (=\otimes \equiv)\xrightarrow{\alpha} (-\otimes =)\otimes \equiv;\hspace{1.5cm}\one\otimes -\xrightarrow{\lambda^l}\Id_{\mathcal{C}};\hspace{1.5cm}-\otimes \one\xrightarrow{\lambda^r}\Id_{\mathcal{C}}.
$$
The 6-tuple $(\mathcal{C},\otimes,\one,\alpha,\lambda^l,\lambda^r)$ is called a {\it monoidal category} if it satisfies the conditions
\begin{multline*}
1_X\otimes \lambda^l_Y=(\lambda^r_X\otimes 1_Y)\circ \alpha_{X,\one,Y}:X\otimes (\one\otimes Y)\rightarrow X\otimes Y;\\
(\alpha_{W,X,Y}\otimes 1_Z)\circ\alpha_{W,X\otimes Y,Z}\circ(1_W\otimes \alpha_{X,Y,Z})
=\alpha_{W,X, Y\otimes Z}\circ \alpha_{W\otimes X,Y,Z}:\\W\otimes\big(X\otimes(Y\otimes Z)\big)\rightarrow \big((W\otimes X)\otimes Y\big)\otimes Z
\end{multline*}
for any $X,Y,Z,W\in\mathcal{C}$ (see \cite[\S1.2]{Hermann2} for the definition illustrated with commutative diagrams). In this case $\otimes$ is called a {\it monoidal product} for $\mathcal{C}$ and $\one$ is the {\it unit} of $\otimes$.
\end{definition}

\begin{remark} Mac Lane's Coherence Theorem (see \cite{CWM}) states that any ``formal''\ diagram involving identity morphisms and isomorphisms $\alpha$, $\lambda^l$ and $\lambda^r$ commutes. Roughly speaking, this means that if we have a sequence $X_1,\dots, X_n$, where each $X_i$ is either the object $\one$ or a formal variable, and a sequence $Y_1,\dots,Y_m$ which is obtained from the first one via exclusion of objects $X_i$ that are equal to $\one$, then any two isomorphisms from $\big(X_1\otimes\cdots\otimes(X_{n-1}\otimes X_n)\cdots\big)$ to $\big(\cdots(Y_1\otimes Y_2)\otimes\cdots\otimes Y_m\big)$ formed by formally defined compositions of morphisms of one of the forms $1^{\otimes a}\otimes\alpha_{A,B,C}\otimes 1^{\otimes b}$, $1^{\otimes a}\otimes\lambda_{A}^l\otimes 1^{\otimes b}$ and $1^{\otimes a}\otimes\lambda_{A}^r\otimes 1^{\otimes b}$ are equal. In particular, $\one^{\ot r}$ is canonically isomorphic to $\one$ for any $r\ge 0$.
\end{remark}

Similarly to the exact category case, we will usually omit the notation $\otimes,\one,\alpha,\lambda^l,\lambda^r$ and call $\mathcal{C}$ a monoidal category meaning that there is some fixed monoidal category structure for it.

Suppose now that $\C$ is monoidal and exact at the same time. Let $(E,\phi)$ and $(F,\psi)$ be two complexes over $\C$. Suppose that either $\C$ admits arbitrary countable direct sums or $E$ and $F$ are {\it bounded below}, i.e. there exists $N\in\mathbb{Z}$ such that $E_i=F_i=0$ for $i<N$. We define their tensor product complex $(E\ot F,\phi\ot\psi)$ in the following way. We set $(E\ot F)_i=\oplus_{j+k=i}E_j\ot F_k$ and $(\phi\ot\psi)_{i-1}|_{E_j\ot F_k}=\phi_{j-1}\ot 1_{F_k}+(-1)^{j}(1_{E_j}\ot \psi_{k-1})$. Unfortunately, one cannot guarantee that $(E\ot F,\phi\ot\psi)$ is really a complex, because the definition of a monoidal category does not require bilinearity of the tensor product. This problem does not arise in the case where $\C$ is a {\it tensor category}, but in fact here it is enough to add the condition $0\ot X\cong 0$ for any object $X$ of $\C$, where $0$ is the zero object.
If $f:E\rightarrow E'$ is a degree $n$ morphism and $g:F\rightarrow F'$ is a degree $m$ morphism, then we define the degree $(n+m)$ morphism $f\ot g:E\ot F\rightarrow E'\ot F'$ by the equality $(f\ot g)_i|_{E_j\ot F_k}=(-1)^{mj}(f_j\ot g_k)$.
If we forget for some time that the tensor product complex does not have to be a complex, then our definitions turn the category of (bounded below) complexes over $\C$ into a monoidal category with the unit object $\one$.
Mac Lane's Coherence Theorem can be applied in this context and so we will always identify tensor products with different bracket arrangements and complexes $E$, $\one\ot E$ and $E\ot \one$ without a special mentioning. In  particular, the notation $E^{\ot r}$ makes sense for $r\ge 0$. Note also that all of our notation is justified in such a way that the Koszul sign convention can be applied, for example, $\del(f\ot g)=\del(f)\ot g+(-1)^nf\ot \del(g)$, etc.

Suppose now that $(E,\phi,\mu_E)$ is a resolution of $X$ and $(F,\psi,\mu_F)$ is a resolution of $Y$. Then the tensor complex $E\ot F$ is equipped with the morphism $\mu_E\ot\mu_F$ and one can ask if $(E\ot F,\phi\ot \psi,\mu_E\ot\mu_F)$ is a resolution of $X\ot Y$. Of course, in general, there is no reason that this should be true.

\begin{definition}\label{def:powerflat}
A resolution $(P,d,\mu_P)$ of $\one$ is called {\it $n$-power flat} if $(P^{\ot r},d^{\ot r},\mu_P^{\ot r})$ is a resolution of $\one$ for each $1\le r\le n$. If $P$ is $n$-power flat for each $n\ge 2$, then we say that $P$ is {\it power flat}.
\end{definition}

The main object of our study is the $\Ext$-algebra $\Ext^{\bu}_{\C} (\one,\one)$ of the unit of a category $\mathcal{C}$ that is exact and monoidal at the same time.
The assumption that we will need to obtain our results is that $\one$ has a projective power flat resolution $P$.

\begin{remark}
Note that in \cite[\S5.2]{Hermann2} the bracket was defined under the condition that, for any $(E,\phi,\mu_E)\in\cExt_{\mathcal{C}}^n(\one,\one)$ and $(F,\psi,\mu_F)\in\cExt_{\mathcal{C}}^m(\one,\one)$, $(E\ot F,\phi\ot\psi,\mu_E\ot\mu_F)$ is an element of $\cExt_{\mathcal{C}}^{n+m}(\one,\one)$. Applying this property to the powers of the $(N+3)$-extension
\[P(N)=\left(0\rightarrow\one\xrightarrow{1_{\one}}\one\xrightarrow{0}\Ker(d_{N-1})\hookrightarrow P_N\xrightarrow{d_{N-1}}\cdots\xrightarrow{d_1}P_1\xrightarrow{d_0}P_0\right)\]
with $\mu_{P(N)}=\mu_P$ for big enough $N$, one can see that the property assumed in \cite[\S5.2]{Hermann2} implies power flatness of {\it any} resolution of $\one$ and our proofs can be applied if $\one$ has a projective resolution.
\end{remark}

Let us now recall some definitions and facts of \cite{NVW} and adapt them to our context. The nice feature of this approach is that, due to Mac Lane's Coherence Theorem, the proofs from \cite{NVW} work without changes and we automatically have a Gerstenhaber algebra structure on $\Ext^{\bu}_{\C} (\one,\one)$. This was difficult to do using the approach of \cite{Hermann2} and was left as a question there (see \cite[Question 5.2.15]{Hermann2}).
In the next sections, we will show that our approach and the approach of \cite{Hermann2}, in the cases where both of them can be applied, give the same operation on $\Ext^{\bu}_{\C} (\one,\one)$ up to a sign. Since the proofs of theorems stated in the remaining part of this section do not differ from the proofs given in \cite{NVW}, we leave all of them to the reader.

\begin{definition} An {\it $A_\infty$-coalgebra} over the exact monoidal category $\C$ is a (bounded below) complex $(C,0)$ with a collection of degree one morphisms $\delta_n:C\rightarrow C^{\otimes n}$, for all $n\ge 1$, such that, for any $N\ge 1$, 
\begin{equation}\label{Ainf}
0=\sum\limits_{r+s+t=N}(1_C^{\otimes r}\otimes\delta_s\otimes 1_C^{\otimes t})\delta_{r+t+1}.
\end{equation}
A degree one map $\mu:C\rightarrow \one$ is called a {\it weak counit} of the $A_\infty$-coalgebra $C$ if $(\mu\otimes\mu)\delta_2=\mu$ and $\mu^{\otimes n}\delta_n=0$ for all $n>2$.
\end{definition}

\begin{remark} Note that formally the targets of the morphisms $(1_C^{\otimes r}\otimes\delta_s\otimes 1_C^{\otimes t})\delta_{r+t+1}$  can be different, but there exist isomorphisms $\phi_{r,s,t}$ that can be expressed as compositions of isomorphisms of the form $1_C^{\otimes a}\otimes \alpha_{X,Y,Z}\otimes 1_C^{\otimes b}$ such that all morphisms $\phi_{r,s,t}(1_C^{\otimes r}\otimes\delta_s\otimes 1_C^{\otimes t})\delta_{r+t+1}$ make sense and have the same target. Moreover, Mac Lane's Coherence Theorem guarantees that the  isomorphism $\phi_{r,s,t}$ does not depend on a concrete choice of composed isomorphisms and their order. This is the reason why \eqref{Ainf} makes sense. In fact, this is an example of an identification of tensor products with different bracket arrangements.
\end{remark}

Suppose that $C$ is an  $A_\infty$-coalgebra as in the definition.
Let $f=(f_n)_{n\ge 0}$ and $g=(g_n)_{n\ge 0}$ be two sequences of morphisms, where, for each $n\ge 0$, the morphism $f_n:C\to C^{\otimes n}$ has degree $l$ and the morphism $g_n:C\to C^{\otimes n}$ has degree $k$. Then we define
$f\circ g=\big((f\circ g)_n\big)_{n\ge 0}$ by the equality

$$
(f\circ g)_n= \sum_{r+s+t=n} (1_C^{\otimes r}\otimes f_{s}\otimes 1_C^{\otimes t})g_{r+t+1} 
$$
and set $[f,g]=f\circ g-(-1)^{kl}g\circ f$. Note that $\delta=(\delta_n)_{n\ge 0}$ with $\delta_0=0$ is a sequence of degree one morphisms satisfying the equality $\delta\circ\delta=0$. For $f$ and $g$ as above, we define also $f\smile g=\big((f\smile g)_n\big)_{n\ge 0}$ by the equality (recall $k$ is the degree of $g$)
$$
(f\smile g)_n= (-1)^k\sum\limits_{r+s+t+u+v=n}
(1^{\ot r}\ot f_s\ot 1^{\ot t}\ot g_u\ot 1^{\ot v})\delta_{r+t+v+2} .
$$

\begin{definition}\label{CDinf}
Let $C$ be an $A_\infty$-coalgebra over $\C$.  A {\it degree $l$ $A_\infty$-coderivation} $f:C\to C$ is defined as a sequence of degree $l$ maps $f_n:C\to C^{\otimes n}$, for $n\geq 0$, that satisfy the equality $[f,\delta]=0$. The degree $l$ $A_\infty$-coderivation $f$ is called {\it inner} if there exists a sequence of degree $(l-1)$ maps $g_n:C\to C^{\otimes n}$, for all $n\geq 0$, such that $f=[g,\delta]$. We will denote by $\Coder^{\infty}_\C(C)$ and $\Inn^{\infty}_\C(C)$ the set of $A_\infty$-coderivations and the set of inner $A_\infty$-coderivations on the object $C$ respectively.
\end{definition}

Now we can reformulate \cite[Theorem 2.4.7]{NVW} in our setting.

\begin{theorem}\label{thm:NVW}
  If $(C,\delta)$ is an $A_{\infty}$-coalgebra over the monoidal category $\C$, then $\Inn^{\infty}_\C(C)$ is an ideal in $\Coder^{\infty}_\C(C)$ with respect to the operations $\smile$ and $[ \ , \ ]$. Moreover, $\Big(\big(\Coder^{\infty}_\C(C)/\Inn^{\infty}_\C(C)\big)[1],\ \smile, \ [ \ , \ ]\Big)$ is a Gerstenhaber algebra (in general, nonunital).
\end{theorem}

\begin{remark}
  Consider the category of bimodules $\C$ of an algebra $A$,
  with tensor product over $A$.
  Take $C$ to be a projective bimodule resolution of $A$ with
  $A_\infty$-coalgebra structure as described in~\cite[Theorem 3.1.1]{NVW}
  (see also Theorem~\ref{GerIso} below).
  Theorem 4.1.1 of~\cite{NVW} states that 
  the Gerstenhaber algebra of Theorem~\ref{thm:NVW} above is precisely the
  Hochschild cohomology algebra of $A$.
  See Theorem~\ref{GerIso} below
  for a restatement in our exact monoidal category setting.
    \end{remark}

Suppose now that $\C$ is an exact monoidal category and $(P,d,\mu_P)$ is a projective power flat resolution of $\one$.
Then there exists a morphism of resolutions $\Delta_P:P\rightarrow P\ot P$.

\begin{remark}\label{Dchoice} In our calculations it will be convenient to justify the choice of $\Delta_P$. Namely, let us introduce $\alpha_P=\lambda^l_P(\mu_P\otimes 1_P)\Delta_P$ and $\beta_P=\lambda^r_P(1_P\otimes \mu_P)\Delta_P$. Then the map $\Delta_P':P\rightarrow P\ot P$ defined by the equality $\Delta_P'=(\alpha_P\otimes 1_P-\beta_P\otimes 1_P)\Delta_P+\Delta_P\beta_P$ is also a  morphism of resolutions that additionally satisfies the equality $\lambda^l_P(\mu_P\otimes 1_P)\Delta_P'=\lambda^r_P(1_P\otimes \mu_P)\Delta_P'$.
\end{remark}

Now \cite[Theorem 3.1.1]{NVW} (see also \cite[Proposition 5.3]{LowVan}) can be transferred to our setting.

\begin{theorem}\label{Pinfc}
  Let $\C$ be an exact monoidal category for which the unit object $\one$
  has a projective power flat resolution $(P,d,\mu_P)$.
The complex $(P[-1],0)$ admits an $A_{\infty}$-coalgebra structure $\delta$ with $\delta_1=d$ and $\delta_2=\Delta_P$ such that $\mu_P$ is a weak counit for $(P[-1],\delta)$.
\end{theorem}

Note that $\big(\Coder^{\infty}_\C(P[-1])/\Inn^{\infty}_\C(P[-1])\big)[1]$ is a graded associative algebra with respect to the product $\smile$ and $\Ext^{\bu}_{\C}(\one,\one)$ is a graded associative algebra with respect to the Yoneda product.
We state our version of \cite[Theorem 4.1.1]{NVW}.

\begin{theorem}\label{GerIso}
  Let $\C$ be an exact monoidal category for which the unit object $\one$
  has a projective power flat resolution $(P,d,\mu_P)$. 
There exists an isomorphism of graded algebras $$\big(\Coder^{\infty}_\C(P[-1])/\Inn^{\infty}_\C(P[-1])\big)[1]\cong \Ext^{\bu}_{\C}(\one,\one)$$
that sends the class of the sequence $f=(f_n)_{n\ge 0}$ in $\Coder^{\infty}_\C(P[-1])/\Inn^{\infty}_\C(P[-1])$ to the class of $f_0$ in $\Ker \Hom_\C(d,\one)/\im \Hom_\C(d,\one)$.
\end{theorem}

As a consequence of 
the isomorphism of Theorem \ref{GerIso}, there is a Gerstenhaber algebra structure on $\Ext^{\bu}_{\C}(\one,\one)$. This structure can be described independently of $A_\infty$-coalgebra techniques via the next definition introduced in \cite[Definition 4.3]{Volkov} (see also~\cite[Section~6.3]{W}).

\begin{definition}\label{def:hl}
Let $f:P\rightarrow\one$ be an $n$-cocycle. 
A degree $(n-1)$ morphism $\psi_f :P\rightarrow P$ is a {\em homotopy lifting of $(f,\Delta_P)$} if 
\begin{equation}\label{eqn:hl1}
    \del (\psi_f) = (f\ot 1_P - 1_P \ot f)\Delta_P 
\end{equation}
and  $\mu_P \psi_f\sim (-1)^{n+1}f\psi$ for some degree $-1$ map $\psi:P\rightarrow P$ such that
$$\del (\psi) = (\mu_P\ot 1_P - 1_P \ot \mu_P)\Delta_P.$$
\end{definition}

\begin{remark}\label{Lchoice} It is easy to see from the definition that $\psi_f$ is defined uniquely up  to homotopy by $f$ and $\Delta_P$.
Moreover, if $\psi_f$ is a homotopy lifting of $(f,\Delta_P)$ and $(\mu_P\ot 1_P)\Delta_P=(1_P\ot \mu_P)\Delta_P$, then one can choose $\psi=0$ in the definition. In this case $\mu_P \psi_f$ is a coboundary, and hence there exists a degree $(m-1)$ null-homotopic chain map $\Phi:P\rightarrow P$ such that $\mu_P\Phi=\mu_P\psi_f$. Then $\psi'_f=\psi_f-\Phi$ is a homotopy lifting of $(f,\Delta_P)$ such that $\mu_P\psi'_f=0$.
\end{remark}

Note that the analog of \cite[Theorem 4.4.6]{NVW} states that the isomorphism of Theorem \ref{GerIso} is induced by a surjective map from $\Coder^{\infty}_\C(P[-1])$ to $\Ker\Hom_\C(d,\one)$ and the analog of \cite[Lemma 4.5.4]{NVW} states that if $f=(f_n)_{n\ge_0}$ is a degree $m$ $A_\infty$-coderivation, then $(-1)^mf_1$ is a homotopy lifting of $(f_0,\Delta_P)$. This argument ensures that the Gerstenhaber bracket coming from Theorem \ref{GerIso} can be calculated in the following way. Let $f,g:P\rightarrow \one$ be an $m$-cocycle and a $k$-cocycle respectively. Let $\psi_f$ and $\psi_g$ be homotopy liftings of $(f,\Delta_P)$ and $(g,\Delta_P)$ respectively. We set
\begin{equation}\label{eqn:bracketdefn}
    [f,g] = f \psi_g - (-1)^{(m-1)(k-1)} g \psi_f .
\end{equation}
It is clear from our discussion that this operation induces an operation on $\Ext^{\bu}_{\C}(\one,\one)$ that does not depend on the choice of homotopy liftings. We next state an analog of \cite[Theorem~4.4]{Volkov} that ensures this operation also does not depend on the choice of $P$ and $\Delta_P$; the proof is essentially the same. 
By Theorem \ref{GerIso}, $\Ext^{\bu}_{\C}(\one,\one)$ with the Yoneda product and the bracket $[ \ , \ ]$ is a Gerstenhaber algebra. Of course, this algebra has the unit represented by $\mu_P:P\rightarrow\one$.

\begin{thm}\label{thm:bracket}
  Let $\C$ be an exact monoidal category for which the unit object $\one$
  has a projective power flat resolution $(P,d,\mu_P)$.
Let $f,g:P\rightarrow \one$ be cocycles. 
The element of $\Ext^{\bu}_{\C}(\one, \one)$ given by 
$[f,g]$ at the cochain level is independent of the choice of a projective resolution
$P$ and of a morphism of resolutions $\Delta_P$.
\end{thm}

In particular, this means that to calculate the bracket $[f,g]$ one can choose $\Delta_P$, $\psi_f$ and $\psi_g$ in such a way that $(\mu_P\ot 1_P)\Delta_P=(1_P\ot \mu_P)\Delta_P$ and $\mu_P\psi_f=\mu_P\psi_g=0$ (see Remarks \ref{Dchoice} and \ref{Lchoice}).

\begin{remark}
Let us recall that, starting from Theorem \ref{Pinfc}, the resolution $P$ of $\one$ is assumed to be projective and power flat.
In fact, we need only $2$-power flatness of $P$ to define the bracket, and we only need $n$-power flatness for some small values of $n$ for the Gerstenhaber algebra structure, 
but a proof would require a generalization of the $A_{\infty}$-coderivation tools
of~\cite{NVW} from bimodules to monoidal categories. 
It is not the aim of this paper and we do not see a big difference between stating the power flatness and stating the $n$-power flatness for small $n$. For example, $P$ is $2$-power flat if $P$ is formed by flat (with respect to $\ot$) objects, but in this case $P$ is power flat as well.
One can also get a strict Gerstenhaber algebra structure on $\Ext^{\bu}_{\C}(\one,\one)$ using the operation $\circ$ on the set of $A_\infty$-coderivations (see the definition of a strict Gerstenhaber algebra in \cite[Definition 4.2.1]{Hermann2}), but we will not do this since, as mentioned above, this would require $A_\infty$-coderivation tools for studying the strict Gerstenhaber algebra structure, and this structure is not discussed in \cite{NVW}.
\end{remark}

\section{Equivalence of different definitions of the bracket}
\label{sec:equiv}

We summarize some of Schwede's and Hermann's exact sequence interpretations of the Lie structure on $\Ext^{\bu}_{\C}(\one,\one)$ (see \cite{Schwede} and \cite{Hermann2}) and prove that up to signs, and under our hypotheses, these give the same operation as our homotopy lifting approach. In particular, Theorem~\ref{thm:coincide} below,
combined with the results of~ Section~\ref{sec:Gb} above,
implies that if  $\one$ has a power flat projective resolution in an exact monoidal category $\C$ and the conditions required in \cite{Hermann2} are satisfied, then the operations on $\Ext^{\bu}_{\C}(\one,\one)$ defined in \cite{Hermann2}  give a structure of a Gerstenhaber algebra,
answering a question left open there.  We will assume in this section that $m,n\ge 1$, since the construction of \cite{Hermann2} is specifically for this case; to study the case where $m$ or $n$ is zero one must inspect the construction of \cite{Hermann1} but we will not do this here.

Consider an $m$-extension $(E,\phi,\mu_E)$ and an $n$-extension $(F,\psi,\mu_F)$ of $\one$ by $\one$. Assume that $(E\otimes F,\phi\ot\psi,\mu_E\otimes\mu_F)$ and $(F\otimes E,\psi\ot\phi,\mu_F\otimes\mu_E)$ are $(m+n)$-extensions of $\one$ by $\one$.
This assumption is necessary for the constructions of Schwede 
and of Hermann.
Suppose also that $\one$ has a projective $2$-power flat resolution $(P,d,\mu_P)$.
There exist morphisms of resolutions $\hat{f}:P\rightarrow E$ and $\hat{g}:P\rightarrow F$. Let us set $f=\hat{f}_m=\pi_E\hat{f}$ and $g=\hat{g}_n=\pi_F\hat{g}$. Then $f$ is an $m$-cocycle corresponding to $E$ and $g$ is an $n$-cocycle corresponding to $F$. So for example $f$ is defined via the following commuting diagram: 
\[
\begin{xy}*!C\xybox{
\xymatrix{
  \cdots \ar[r]& P_m\ar[r]^{d_m}\ar[d]^{f} & P_{m-1}\ar[r]^{d_{m-1}}\ar[d]^{\hat{f}_{m-1}} 
    & \cdots\ar[r]^{d_1} & P_0\ar[r]^{\mu_P}\ar[d]^{\hat{f}_0} & \one\ar[r]\ar[d]^{=} & 0 \\
    0\ar[r]& \one\ar[r]^{\iota_E} & E_{m-1}\ar[r]^{\phi_{m-2}} &\cdots\ar[r]^{\phi_0} & E_0\ar[r]^{\mu_E} & \one\ar[r]& 0
}}
\end{xy}
\]

We fix a morphism of resolutions $\Delta_P: P\rightarrow P\ot P$ and set $f\smile g=(-1)^{mn}(f\otimes g)\Delta_P$.
Then each of the $(m+n)$-extensions $E\# F$, $E\ot  F$, $(-1)^{mn} F\# E$ and $(-1)^{mn} F\ot E$ is represented by $f\smile g$.
Let us recall that if $(L,\chi,\mu_L)$ is an extension, then $-L$ denotes the extension $(L,\chi,-\mu_L)$.
The fact that these four extensions are all equivalent is
depicted in the following diagram involving four specific morphisms defined below: 
\begin{equation}\label{eqn:diamond}
\begin{xy}*!C\xybox{
\xymatrix{
             &      E\ot F\ar[dr]^{\rho_{E,F}}\ar[dl]_{\lambda_{E,F}} & \\
  E\# F && (-1)^{mn} F\# E \\
   & (-1)^{mn} F\ot  E\ar[ul]^{\rho_{F,E}}\ar[ur]_{\lambda_{F,E}} & 
}}
\end{xy}
\end{equation}

To define the morphisms $\lambda_{E,F}$, $\lambda_{F,E}$, $\rho_{E,F}$
and $\rho_{F,E}$,  
consider the augmented double complex:
\[
\begin{xy}*!C\xybox{
\xymatrix{
  & E_0\ar[dl]\ar[d] & E_1\ar[l]\ar[d] & \cdots\ar[l] & E_{m-1}\ar[l]\ar[d] & \one\ar[l]\ar[d]\\
  F_{n-1}\ar[d] & E_0\ot F_{n-1}\ar[l]^{\mu_E\otimes 1_{F_{n-1}}}\ar[d] & E_1\ot F_{n-1}\ar[l]\ar[d] & \cdots\ar[l] 
        & E_{m-1}\ot  F_{n-1}\ar[l]\ar[d] & F_{n-1}\ar[l]\ar[d] \\
  \vdots\ar[d] & \vdots\ar[d] & \vdots\ar[d] & & \vdots\ar[d] & \vdots\ar[d] \\
  F_1\ar[d] & E_0\ot F_1\ar[l]^{\mu_E\otimes 1_{F_{1}}}\ar[d] & E_1\ot F_1\ar[l]\ar[d] & \cdots\ar[l] &
       E_{m-1}\ot  F_1\ar[l]\ar[d] & F_1\ar[l]\ar[d] \\
  F_0\ar[d] & E_0\ot F_0\ar[l]^{\mu_E\otimes 1_{F_{0}}}\ar[d]_{1_{E_0}\ot \mu_F} & E_1\ot  F_0\ar[l]\ar[d]^{1_{E_1}\ot \mu_F} & \cdots\ar[l] &
    E_{m-1}\ot  F_0\ar[l]\ar[d]_{1_{E_{m-1}}\ot \mu_F} & F_0\ar[l]\ar[dl] \\
  \one & E_0\ar[l] & E_1\ar[l] & \cdots\ar[l] & E_{m-1}\ar[l] &
}}
\end{xy}
\]
All but the leftmost column and bottom row constitute 
the double complex with totalization $E\ot  F$,
and the outermost rows and columns are  $E\# F$ (left column and top row)
and $(-1)^{mn} F\# E$
(right column and bottom row).
Now let us pick in general a complex $L$.
To define a chain map $\rho:L\rightarrow E\#F$ one has to define a degree $n$ chain map $\rho^1:L\rightarrow E$ and a degree zero chain map $\rho^0:L\rightarrow F$ such that $\mu_E\rho^1=\pi_F\rho^0$.
Given the pair $(\rho^1,\rho^0)$, the morphism $\rho$ can be recovered by the equalities $\rho_i=\rho^0_i$ for $0\le i\le n-1$ and $\rho_i=(-1)^{n(i-n)}\rho^1_i$ for $n\le i\le m+n$.
In these terms, the maps $\lambda_{E,F}$ and $\rho_{E,F}$ are defined by the pairs $(1_E\otimes\pi_F,\mu_E\otimes 1_F)$ and $\big((-1)^{mn}\pi_E\ot 1_F,(-1)^{mn}1_E\ot\mu_F\big)$. Then $\lambda_{E,F}$ and $\rho_{E,F}$ are morphisms of extensions.
Similarly there are morphisms $\rho_{F,E}: (-1)^{mn} F\ot E\rightarrow E\# F$
and $\lambda_{F,E}: (-1)^{mn} F\ot E\rightarrow (-1)^{mn} F\# E$.

\begin{remark} Some additional signs in definitions appear because of the not very natural construction of the Yoneda product. Actually, during the construction of $E\# F$ we use $E[n]$ instead of $E$ and so it would be natural to replace $\phi$ by $(-1)^n\phi$. We have not done this because of some classical traditions concerning the definition of the Yoneda product.
\end{remark}

Diagram~(\ref{eqn:diamond}) represents a loop in the extension
category $\cExt^{m+n}_\C(\one,\one)$.
By results of Retakh and Neeman \cite[Theorem 5.2]{NeRe},  \cite[Theorem 1]{Retakh}, the homotopy classes of such loops are in one-to-one correspondence with $\Ext^{m+n-1}_\C(\one,\one)$.
This was proven by Hermann~\cite{Hermann1} for factorizable exact categories;
by our Lemma~\ref{lem:factorizing}, any exact category is factorizing.
By a result of Schwede~\cite[Theorem 3.1]{Schwede}, in the case where $\C$ is the category of $A$-bimodules with $\otimes=\otimes_A$, the loop \eqref{eqn:diamond} corresponds up to some sign to the Gerstenhaber bracket $[f,g]$.
We refer to Retakh~\cite{Retakh} and Schwede~\cite{Schwede} for details. On the other hand, in an arbitrary monoidal category, Hermann defined the bracket operation using the loop \eqref{eqn:diamond} and the isomorphism $\gamma$ from Corollary \ref{monoRek} (see the text after Lemma~\ref{preim} for the description of $\gamma$).
Here we adapt Schwede's proof to show that this loop indeed corresponds up to a sign to the cocycle $[f,g]$ defined by the equality \eqref{eqn:bracketdefn}. Note that by Lemma~\ref{SchHer}, we may replace $\gamma$ by $\mu$.

\begin{thm}\label{thm:coincide}
  Let $\C$ be an exact monoidal category for which the unit object $\one$ has
  a power flat resolution $(P,d,\mu_P)$. 
Suppose that $E$ and $F$ are an $m$-extension and an $n$-extension of $\one$ by $\one$ such that $E\ot F,F\ot E\in \cExt^{m+n}_{\C}(\one,\one)$.
Let $\Delta_P:P\rightarrow P\ot P$ be a morphism of resolutions.
 Let $f$ and $g$ be cocycles representing the classes of $E$ and $F$ in $\Ext^{\bu}_{\C}(\one,\one)$, respectively, and let $\psi_f$ and $\psi_g$ be homotopy liftings of $(f,\Delta_P)$ and $(g,\Delta_P)$.
 Then the cocycle $(-1)^m[g,f]$ defined by Equation~\eqref{eqn:bracketdefn} represents $\mu^{-1}(\rho_{F,E}^{-1}\lambda_{E,F}\rho_{E,F}^{-1}\lambda_{F,E})\in \Ext^{m+n-1}_{\C}(\one,\one)$, where $\mu$ is the isomorphism~\eqref{eqn:mu} and the loop
 $\rho^{-1}_{F,E}\lambda_{E,F}\rho^{-1}_{E,F}\lambda_{F,E}$
 can be seen in Diagram~\eqref{eqn:diamond}.
\end{thm}
\begin{proof}
  We plan to apply Lemma~\ref{preim} to a suitable homotopy
  $s: P\rightarrow E\# F$ in relation to the two outermost compositions
  of maps in Diagram~\eqref{eqn:projchange} below.
  We will relate the homotopy $s$ to the cocycle $(-1)^m[g,f]$,
  and we will relate those outermost compositions
  to the loop in the theorem statement. 
  We begin by fixing some choices of maps and explaining the diagram. 

  We may choose $\Delta_P$, $\psi_f$ and $\psi_g$ in such a way that $(\mu_P\ot 1_P)\Delta_P=(1_P\ot \mu_P)\Delta_P$ and $\mu_P\psi_f=\mu_P\psi_g=0$ (see Theorem \ref{thm:bracket} and the sentence after it).

We may assume that we have morphisms of resolutions $\hat{f}:P\rightarrow E$ and $\hat{g}:P\rightarrow F$ such that $f=\pi_E\hat{f}$ and $g=\pi_F\hat{g}$. Then $(\hat{f}\otimes \hat{g})\Delta_P:P\rightarrow E\ot F$ and $(-1)^{mn}(\hat{g}\otimes \hat{f})\Delta_P:P\rightarrow (-1)^{mn}F\ot E$ are also morphisms of resolutions. Our first aim is to construct a chain map $\varepsilon:P\rightarrow (-1)^{mn}F\ot E$ satisfying $(\mu_F\ot\mu_E)\varepsilon=0$ that makes the rightmost quadrilateral in the following diagram commute, allowing us to replace the
loop $\rho^{-1}_{F,E}\lambda_{E,F}\rho^{-1}_{E,F}\lambda_{F,E}$
by one involving the resolution $P$. 
\begin{equation}\label{eqn:projchange}
\begin{xy}*!C\xybox{
\xymatrix{    & E\ot F\ar[dl]_{\lambda_{E,F}}\ar[dr]_{\rho_{E,F}} & & \\   
    E\# F && (-1)^{mn}F\# E &  P\ar[ull]_{(\hat{f}\otimes \hat{g})\Delta_P}
      \ar[dll]^{\hspace{0.5cm}(-1)^{mn} (\hat{g}\otimes \hat{f})\Delta_P + \varepsilon}  \\
    &  (-1)^{mn} F\ot E \ar[ul]^{\rho_{F,E}}\ar[ur]^{\lambda_{F,E}}  && 
}}
\end{xy}
\end{equation}
Note that the universal property of pushout implies existence of unique morphisms $\bar\alpha:K(f\smile g)_{m+n-1}\rightarrow (E\ot F)_{m+n-1}$ and $\bar\beta:K(f\smile g)_{m+n-1}\rightarrow (F\ot E)_{m+n-1}$ such that 
\begin{multline*}
\bar\alpha\theta_{f\smile g}=\big((\hat{f}\otimes \hat{g})\Delta_P\big)_{m+n-1}, \bar\alpha\iota_{f\smile g}=\iota_{E\ot F}, \\
\bar\beta\theta_{f\smile g}=\big((-1)^{mn}(\hat{g}\otimes \hat{f})\Delta_P\big)_{m+n-1}+\epsilon_{m+n-1}, \bar\beta\iota_{f\smile g}=\iota_{F\ot E}.
\end{multline*}
Hence, there are unique morphisms $\alpha:K(f\smile g)\rightarrow E\ot F$ and $\beta:K(f\smile g)\rightarrow (-1)^{mn}F\ot E$ of $(m+n)$-extensions that satisfy the equalities $\alpha\Phi_{f\smile g}=(\hat{f}\otimes \hat{g})\Delta_P$ and $\beta\Phi_{f\smile g}=(-1)^{mn}(\hat{g}\otimes \hat{f})\Delta_P+\varepsilon$, where $\Phi_{f\smile g}: P\rightarrow K(f\smile g)$ is the chain map defined just before Lemma \ref{preim}.
Another application of the pushout universal property implies that $\rho_{E,F}\alpha=\lambda_{F,E}\beta$, and hence the loop $\rho_{F,E}^{-1}\lambda_{E,F}\rho_{E,F}^{-1}\lambda_{F,E}$ is homotopic to the loop $\beta^{-1}\rho_{F,E}^{-1}\lambda_{E,F}\alpha$ up to conjugation. We will show there is a homotopy between $\lambda_{E,F}(\hat{f}\otimes \hat{g})\Delta_P$ and $\rho_{F,E}\big((-1)^{mn}(\hat{g}\otimes \hat{f})\Delta_P+\varepsilon\big)$, and obtain $\mu^{-1}(\rho_{F,E}^{-1}\lambda_{E,F}\rho_{E,F}^{-1}\lambda_{F,E})$ using Lemma \ref{preim}.

Before we do that, 
we want to find a chain map $\varepsilon$ such that the morphism $\lambda_{F,E}\varepsilon$ defined by the pair of morphisms $\big((1_F\ot \pi_E)\varepsilon, (\mu_F\ot 1_E)\varepsilon\big)$ is equal to $\Psi=\rho_{E,F}(\hat{f}\otimes \hat{g})\Delta_P-(-1)^{mn}\lambda_{F,E}(\hat{g}\otimes \hat{f})\Delta_P$. The morphism $\Psi$ is defined by the pair of morphisms $(\Psi^1,\Psi^0)$, where
\begin{multline*}
\Psi^0=(-1)^{mn}\Big((1_E\ot\mu_F)(\hat{f}\otimes \hat{g})\Delta_P-(\mu_F\ot 1_E)(\hat{g}\otimes \hat{f})\Delta_P
\Big)\\
=(-1)^{mn}(\hat{f}\ot\mu_P-\mu_P\ot \hat{f})\Delta_P=(-1)^{mn}\hat{f}(1_P\ot\mu_P-\mu_P\ot 1_P)\Delta_P=0
\end{multline*}
and
\begin{multline*}
\Psi^1=(-1)^{mn}\Big((\pi_E\ot 1_F)(\hat{f}\otimes \hat{g})\Delta_P-(1_F\ot \pi_E)(\hat{g}\otimes \hat{f})\Delta_P
\Big)\\
=(-1)^{mn}(f\otimes \hat{g}-\hat{g}\otimes f)\Delta_P=(-1)^{mn}\hat{g}(f\otimes 1_P-1_P\otimes f)\Delta_P.
\end{multline*}
Let us set
$$
\varepsilon=(-1)^{mn}\big((\hat{g}\ot\kappa_E)(f\otimes 1_P-1_P\otimes f)\Delta_P+(-1)^m(\hat{g}\ot\iota_E\kappa_E)\psi_f\big).
$$
Note that $\hat{g}\ot\kappa_E$ means here $(\hat{g}\ot\kappa_E)(\lambda^r)^{-1}$ while $f\otimes 1_P$ and $1_P\otimes f$ as usual mean $\lambda^r(f\otimes 1_P)$ and $\lambda^l(1_P\otimes f)$ correspondingly, where $\lambda^r$, $\lambda^l$ are the natural isomorphisms of Definition~\ref{defn:moncat}.
Thus we have found the required chain map $\varepsilon$. 

Now we aim for a homotopy between $\lambda_{E,F}(\hat{f}\otimes \hat{g})\Delta_P$ and $\rho_{F,E}\big((-1)^{mn} (\hat{g}\otimes \hat{f})\Delta_P+\varepsilon\big)$, i.e.\ we want to find a degree $-1$ map $s:P\rightarrow E\#F$ such that
\[
\del(s)=\Gamma=\rho_{F,E}\big((-1)^{mn} (\hat{g}\otimes \hat{f})\Delta_P+\varepsilon\big)-\lambda_{E,F}(\hat{f}\otimes \hat{g})\Delta_P.
\]
The morphism $\Gamma$ is defined by the pair of morphisms $(\Gamma^1,\Gamma^0)$, where
\begin{multline*}
\Gamma^0=(1_F\ot \mu_E)\big((\hat{g}\otimes \hat{f})\Delta_P+(\hat{g}\ot\kappa_E)(f\otimes 1_P-1_P\otimes f)\Delta_P+(-1)^m(\hat{g}\ot\iota_E\kappa_E)\psi_f\big)\\
-(\mu_E\ot 1_F)(\hat{f}\otimes \hat{g})\Delta_P=\hat{g}(1_P\ot\mu_P-\mu_P\ot 1_P)\Delta_P=0
\end{multline*}
since $\mu_E\kappa_E=\mu_E\iota_E=0$, and
\begin{multline*}
\Gamma^1=(\pi_F\ot 1_E)\big((\hat{g}\otimes \hat{f})\Delta_P+(\hat{g}\ot\kappa_E)(f\otimes 1_P-1_P\otimes f)\Delta_P+(-1)^m(\hat{g}\ot\iota_E\kappa_E)\psi_f\big)\\
-(1_E\ot \pi_F)(\hat{f}\otimes \hat{g})\Delta_P=\hat{f}(g\ot 1_P-1_P\ot g)\Delta_P\\
+(-1)^{mn}\kappa_Eg(f\otimes 1_P-1_P\otimes f)\Delta_P+(-1)^{m(n-1)}\iota_E\kappa_Eg\psi_f.
\end{multline*}

Note that $\bar s=\hat{f}\psi_g+(-1)^{mn+m+n}\kappa_Eg\psi_f:P\rightarrow E$ is a degree $(n-1)$ map such that $\del(\bar s)=\Gamma^1$ and $\mu_E\bar s=0$. Then $\bar s$ determines the required homotopy $s$ by the equalities $s_i=0$ for $0\le i\le n-1$ and $s_i=(-1)^{n(i-n)}\bar s_i$ for $n\le i\le m+n-1$. Thus, 
\[
s_{m+n-1}=(-1)^{n(m-1)}\bar s_{m+n-1}=(-1)^{m}g\psi_f+(-1)^{n(m-1)}f\psi_g=(-1)^m[g,f].
\]
It follows, by Lemma~\ref{preim}, that the cocycle $(-1)^m{[} g, f {]}$
defined by Equation~\eqref{eqn:bracketdefn} represents the inverse image,
under the isomorphism $\mu$, of the loop
$\rho^{-1}_{F,E}\lambda_{E,F}\rho^{-1}_{E,F}\lambda_{F,E}$. 

\end{proof}

\end{document}